\documentclass{amsart}
\usepackage{amssymb}
\usepackage{pb-diagram}

\usepackage[textsize=tiny]{todonotes}

\newtheorem{theorem}{Theorem}[section]
\newtheorem{thm}[theorem]{Theorem}
\newtheorem{prop}[theorem]{Proposition}

\newtheorem{claim}[theorem]{Claim}
\newtheorem{fact}[theorem]{Fact}
\newtheorem{cor}[theorem]{Corollary}

\newtheorem{lemma}[theorem]{Lemma}

\newtheorem{question}[theorem]{Question}

\theoremstyle{definition}
\newtheorem{defn}[theorem]{Definition}

\theoremstyle{remark}

\newtheorem{remark}[theorem]{Remark}

\DeclareMathOperator{\cl}{cl}
 \DeclareMathOperator{\intr}{int}

\newcommand{\la}{\langle}
\newcommand{\ra}{\rangle}
\newcommand{\CM}{{\cal M}}
\newcommand{\CN}{{\cal N}}

\newcommand{\sub}{\subseteq}

\newcommand{\es}{\ensuremath{\emptyset}}

\newcommand{\bbm}[1]{\ensuremath{\mathbf{#1}}}

\newcommand{\mbf}[1]{\ensuremath{\mathbf{#1}}}
\newcommand{\cal}[1]{\ensuremath{\mathcal{#1}}}

\newcommand{\Lrarr}{\ensuremath{\Leftrightarrow}}
\newcommand{\Rarr}{\ensuremath{\Rightarrow}}

\newcommand{\rarr}{\ensuremath{\rightarrow}}

\newcommand{\res}{\ensuremath{\upharpoonright}}

\newcommand{\ve}{\ensuremath{\varepsilon}}

\newcommand{\sm}{\setminus}

\newcommand{\N}{\mathbb{N}}

\def\cl{\mathrm{cl}}
\def\CM{\mathcal M}
\def\CN{\mathcal N}

\title[Groups definable  in
weakly o-minimal non-valuational structures]{Groups definable in
weakly o-minimal non-valuational structures}


\date{\today}

\begin{document}

\author {Pantelis  E. Eleftheriou}
\address{Department of Mathematics and Statistics, University of Konstanz, Box 216, 78457 Konstanz, Germany}
\email{panteleimon.eleftheriou@uni-konstanz.de}



\thanks{Research supported by a Research Grant from the German Research Foundation (DFG) and a Zukunftskolleg Research Fellowship.}

\begin{abstract}

Let $\cal M$ be a weakly o-minimal non-valuational structure, and $\cal N$ its canonical o-minimal extension (by Wencel). We prove that every group $G$ definable in \cal M is a subgroup of a group $K$ definable in \cal N, which is canonical in the sense that it is the smallest such group. As an application, we obtain that $G^{00}= G\cap K^{00}$, and establish Pillay's Conjecture in this setting: $G/G^{00}$, equipped with the logic topology, is a compact Lie group, and if $G$ has finitely satisfiable generics, then $\dim_{Lie}(G/G^{00})= \dim(G)$.



\end{abstract}

\maketitle

\section{Introduction}

Definable groups have been at the core of model theory for at least a period of three decades, largely because of their prominent role in important applications of the subject, such as Hrushovski's proof of the function field Mordell-Lang conjecture in all characteristics (\cite{hr}). Examples include  algebraic groups (which are definable in algebraically closed fields) and  real Lie groups (which are definable in o-minimal structures). Groups definable in o-minimal structures are well-understood. The starting point was Pillay's  theorem in [40] that every such group admits a definable manifold topology that makes it into a topological group, and the most influential work in the area has arguably been the solution of Pillay's Conjecture over a field (\cite{hpp}), which brought to light new tools in  theories with NIP (not the independence property). While substantial work on NIP groups has since been done (for example,  in \cite{cs}), a full description of definable groups is still missing in many broad NIP settings. In this paper, we provide this description for groups definable in weakly o-minimal non-valuational structures, and as an application we establish Pillay's  Conjecture in this setting.



We recall that a structure  $\CM=\la M, <, \dots\ra$ is  \emph{weakly o-minimal} if $<$ is a dense linear order and every definable subset of $M$ is a finite union of convex sets. Weakly o-minimal structures were introduced by Cherlin-Dickmann \cite{cd} in order to study the model theoretic properties of real closed rings. They were later also used in Wilkie's proof of the o-minimality of real exponential field \cite{WilkiePfaff}, as well as in van den Dries' study of Hausdorff limits \cite{vdDriesLimits}. Macpherson-Marker-Steinhorn \cite{MacMaSt}, followed-up by Wencel \cite{wen-nonv, wen-scd}, began a systematic study of weakly o-minimal groups and fields, revealing many similarities with the o-minimal setting.

An important dichotomy between valuational structures -- those admitting a definable proper non-trivial convex subgroup --  and non-valuational ones arose, supported by good evidence that the latter structures resemble o-minimal structures more closely than what the former ones do. For example, strong monotonicity and strong cell decomposition theorems were proved for non-valuational structures in \cite{wen-nonv}. In the same reference,  the \emph{canonical o-minimal extension}  $\cal N$ of $\CM$ was introduced, which is an o-minimal structure whose domain is the Dedekind completion $N$ of $M$, and whose induced structure on $M$ is precisely $\CM$. A further description of $\cal N$ was recently provided in \cite{bhp}.
As far as definable groups are concerned, the existence of a definable group manifold topology was proved in \cite{wen-gps}, extending the aforementioned theorem by Pillay \cite{pi-gps} for o-minimal structures, whereas in the special case when $\cal M$ is the trace of a real closed field on a dense elementary subfield, Baro and Martin-Pizarro \cite{bm} showed that every definable group admits, locally, an algebraic group law.

The main result of this paper is that every group $G$ definable in \cal M is a dense subgroup of a group $K$ definable in $\CN$, which is canonical in the sense that it is the smallest such group (Theorems \ref{main}). The existence of group topology from \cite{wen-gps} and the local algebraic result from \cite{bm} are  straightforward consequences.
Note however that the current setting is much richer than that of o-minimal traces (by \cite{ehk}), and that the result herein is global. Moreover, a far reaching application is obtained: we give a sharp description of the smallest type-definable subgroup $G^{00}$ of $G$ of bounded index, and establish Pillay's Conjecture in this setting (Theorem \ref{app}). 
Finally, all our results apply yet to another category of definable groups, namely \emph{small groups} in certain dense pairs $\la \cal N, P\ra$ (see  Remark \ref{rmk-small} below).

For the rest of this paper, $\cal M$ denotes a weakly o-minimal non-valuational structure expanding an ordered group, and $\cal N$ its canonical o-minimal extension by Wencel \cite{wen-nonv} (some details are given in Section \ref{sec-wom}).
Unless stated otherwise, by `definable' we mean `definable in \cal M with parameters', and by `\cal N-definable', we mean `definable in \cal N with parameters'.
If $X\sub N^n$, $\cl(X)$ denotes the closure of $X$ in the usual order topology. If $K$ is a group definable in $\cal N$, and $X\sub K$, we denote by $\cl_K(X)$ the closure of $X$ in the group topology (given by \cite{pi-gps}).
For every set $X\sub N^n$, we define the \emph{dimension of $X$} as the maximum $k$ such that some projection onto $k$ coordinates contains an open subset of $M^k$ (where $M^0=\{0\}$). We define $\dim(\emptyset)=-\infty$. This definition of dimension is natural in our setting, since, if $X\sub M^n$ is definable, then $\dim X$ equals the usual dimension in the ordered structure \cal M, and if $X\sub N^n$ is $\CN$-definable, then $\dim X$ equals the usual o-minimal dimension (easily, by \cite{wen-nonv} -- see also Section \ref{sec-wom}).



\begin{theorem}\label{main} Let $\cal M$ be a weakly o-minimal non-valuational expansion of an ordered group, and $G$ a definable group. Let $\cal N$ be the canonical o-minimal extension of $\cal M$. Then there is an $\cal N$-definable  group $K$ that contains $G$ as a subgroup, and such that for every other $\cal N$-definable group $H$ that contains $G$ as a subgroup, $K$ $\cal N$-definably embeds in $H$. Moreover, $\cl_K(G)=K$ and $\dim G=\dim K$.
\end{theorem}




Pillay's Conjecture was stated in \cite{pi-conj} for groups definable in o-minimal structures, and was established, by splitting different cases and building on previous work, in \cite{emprt, elst, hpp, pet-sbd}. A version of it in the $p$-adic setting was further proved in \cite{op}. Before stating Pillay's Conjecture, let us fix some terminology. Let $\cal R$ be any sufficiently saturated structure. By a \emph{small} set or a set of \emph{bounded cardinality} we mean a set of cardinality smaller than $|R|$. By a type-definable set we mean an intersection of a bounded collection of definable sets. Let $G$ be a group definable in \cal R,  $H$ a normal type-definable subgroup of $G$, and $\pi:G\to G/H$ the canonical group homomorphism. We say that $A\sub G/H$ is closed in the \emph{logic topology} if and only if $\pi^{-1}(A)$ is type-definable.

$ $\\
\noindent\textbf{Pillay's Conjecture.} \emph{Let $G$ be a definably compact group, definable in a sufficiently saturated o-minimal structure. Then:
\begin{enumerate}
  \item $G$ has a smallest type-definable subgroup $G^{00}$ of bounded index.
  \item $G/G^{00}$, equipped with the logic topology, is a compact real Lie group.
  \item If $G$ is definable compact, then $\dim_{Lie} G/G^{00}=\dim G$.
\end{enumerate}}\vskip.2cm
\noindent Notes: (a) $G^{00}$ is necessarily normal, (b) see \cite{petst} for the notion of definable compactness.


If $\cal R$ is a sufficiently saturated NIP structure, it is also known that every definable group $G$ contains a smallest type-definable subgroup $G^{00}$ of bounded index (Shelah \cite{shelah}), and the notion of definable compactness was generalized to the notion of having \emph{finitely satisfiable generics} ($fsg$) in \cite{hpp} (see Definition \ref{def-fsg} below for an alternative definition). As an application of Theorem \ref{main}, we extend Pillay's Conjecture to the weakly o-minimal setting. To our knowledge, this is the first time that the conjecture is being considered in structures that properly expand an o-minimal one.
Recall that \cal M is our fixed weakly o-minimal non-valuational expansion of an ordered group. It is known that $\cal M$ is NIP (\cite[A.1.3]{simon-book}). As remarked in Section \ref{sec-app} below, if \cal M is sufficiently saturated, then so is \cal N in an appropriate signature. 

\begin{theorem}\label{app} Assume $\cal M$ is sufficiently saturated.
Let $G$ be a definable group, and $K$ as in Theorem \ref{main}. Then
 \begin{enumerate}
 \item $G^{00}=G\cap K^{00}$.
 \item $G/G^{00}\cong K/K^{00}$, and hence is a compact real Lie group.
 \item If $G$ has $fsg$, then so does $K$, and hence   $\dim_{Lie} G/G^{00}=\dim G$.
 \end{enumerate}
\end{theorem}


\noindent\textbf{Outline of the strategy.} Let $G$ be a definable group. To show Theorem \ref{main}, we first prove in Section \ref{sec-gpchunk} a group chunk theorem for o-minimal structures, recasting \cite{vdd-weil} with regard to dimension properties instead of topological ones.
In Section \ref{sec-wom}, we establish the existence of an o-minimal group chunk in which $G$ is dense, from which Theorem \ref{main} follows. In Section \ref{sec-app}, we prove Theorem \ref{app}.

\vskip.2cm\noindent\textbf{Terminology/notation.} Let \cal R be a  structure, such as our fixed \cal M or \cal N. Given a set $Z\subseteq R^{2n}$ and $x\in R^n$, we denote
$$Z_x=\{ b \in R^n \ : \ (x, b) \in Z\}$$
and
$$Z^x=\{ a \in R^n \ : \ (a, x) \in Z\}.$$
By $\pi_1, \pi_2: R^{2n}\to R^n$ we denote the projections onto the first and second set of $n$ coordinates, respectively. 
An \emph{embedding} $\sigma :G\to H$ between two groups is simply an injective group homomorphism.



\vskip.2cm\noindent\textbf{Acknowledgments.} I wish to thank E. Baro, A. Hasson and Y. Peterzil for several discussions on the topics of this paper.  I am especially thankful to A. Hasson for pointing out various corrections on earlier drafts, and for suggesting the current proof of Lemma \ref{fsg}.

\section{A group chunk theorem}\label{sec-gpchunk}

The goal of this section is to show a group chunk theorem in the spirit of \cite{vdd-weil}, but with respect to dimension properties instead of topological ones. It is possible that our Theorem \ref{thm-gpchunk} below for \cal N reduces to that reference, however, such a reduction appears to be non-trivial and hence we present a full proof. Moreover, the present account goes through in an abstract dimension-theoretic setting, which we fix next.

For the rest of this section, $\cal R$ denotes a structure that eliminates imaginaries, and `definable' means `definable in \cal R with parameters'. We assume that there is a map $\dim$ from the class  of all definable sets  to $\{-\infty\}\cup \N$ that satisfies the following properties: for all definable $X,Y\sub R^n$,  and $a\in R$,
\begin{enumerate}
\item[\textbf{(D1)}]  $\dim\{a\}=0$, $\dim R=1$, and $\dim X=-\infty$ if and only if $X=\es$

\item[\textbf{(D2)}]  $\dim (X\cup Y)=\max \{\dim X, \dim Y\}$
\item[\textbf{(D3)}]  if $\{X_t\}_{t\in I}$ is a definable family of sets, then
\begin{enumerate}
\item for $d\in \{-\infty\}\cup \N$, the set $I_d=\{t\in I: \dim X_t=d\}$ is definable, and
\item  if every $X_t$ has dimension $k$, and the family is disjoint, then
$$\dim \bigcup_{t\in I} X_t=\dim I + k$$
\end{enumerate}
\item[\textbf{(D4)}]  if $f:X\to Y$ is a definable bijection, then $\dim X=\dim Y$.
\end{enumerate}
For example, \cal N with the usual o-minimal dimension satisfies the above properties (\cite{vdd-book}). Below we will be using the above properties without specific mentioning. Given definable sets $V\sub X\sub R^n$, we call $V$ \emph{large} in $X$ if $\dim(X\sm V)<\dim X$.

\begin{defn} A \emph{definable group chunk} is a triple $(X, i, F)$, where $X\sub R^n$ is a definable set, and $i:X\to R^n$, $F:Z\sub X^2\to R^n$ are definable maps,  such that
\begin{enumerate}
\item 
$i$ is injective on a large subset of $X$.

  \item for every $x\in X$, $F(x,-):Z_x\to F(x,Z_x)$ and  $F(-, x):Z^x\to F(Z^x,x)$ are bijections between large subsets of $X$,

 \item  for every $(x, y)\in Z$, there is a large $S_{(x,y)}\sub X$, such that for every $z\in S_{(x,y)}$, the following expressions are defined and are equal:
\begin{align}
  & (a)\,\,\,\, F(F(x,y), z)=F(x, F(y,z))\notag \\
  &(b)\,\,\,\, F(x, F(i(x), z))=z=F(F(z,x), i(x)),\notag
\end{align}
 \item  for every $x\in X$, $\pi_1 F^{-1}(x)$ and $\pi_2 F^{-1}(x)$ are large in $X$.
\end{enumerate}
\end{defn}

We often refer to properties (1)-(4) above as ``Axioms".

\begin{thm}\label{thm-gpchunk} Let $(X, i, F)$ be a definable group chunk.
Then:
\begin{enumerate}
  \item[(i)] there is a definable group $K=\la K, *, \bbm 1_K\ra$ with $X\sub K$ large in $K$, and such that for every $(x,y)\in Z$,
$$F(x,y)=x * y.$$
\item[(ii)] if $H=\la H, \oplus, \bbm1_H\ra$ is a definable group and $\sigma:X\to H$  an $\cal L$-definable injective map such that for every $(x,y)\in Z$,
    $$\sigma F(x,y)=\sigma(x) \oplus \sigma(y),$$ then $K$ definably embeds in $H$.
    \item[(iii)] if, moreover, $\sigma(X)$ is large in $H$, then $H$ and $K$ are definably isomorphic.
\end{enumerate}
\end{thm}
\begin{proof}   Note that by Axiom (3), for every $a,b\in X$, the set of elements $z\in X$ for which $F(a, F(b,z))$ is defined is large in $X$. 
We will be using this fact without mentioning. Also, we will sometimes write $ab$ for $F(a,b)$ to lighten the notation. Finally, let $k=\dim X$.\\

\noindent  (i) It is enough to find a definable group $K=\la K, *, \bbm{1}_K\ra$ and a   definable injective map $h: X\to K$, with $h(X)$ large in $K$, and such that for every $(x,y)\in Z$,
$$h(F(x,y))=h(x) * h(y).$$
Indeed, then, consider the bijection that maps $x$ to $h^{-1}(x)$, if $x\in h(X)$, and is the identity on $K\sm h(X)$. The induced group structure on $X\cup (K\sm h(X)$ has the desired properties.


For $a\in X$, denote by $F_a$ the map $F(a,-): X\to X$. Define a relation $\sim$ on $X^2$, as follows:$$(a, b) \sim (c,d) \,\,\Lrarr\,\, \text{ $F_a\circ F_b$ agrees with $F_c\circ F_d$ on a large subset of $X$}.$$ By elimination of imaginaries,  $\sim$ is a definable equivalence relation.
Let $K\sub X^2$ be a definable set of representatives for $\sim$. We denote by $[(a,b)]$ the equivalence class of $(a,b)$, and by $[a,b]$ its representative in $K$. We aim to equip $K$ with a definable group structure $\la K, *, \bbm 1_K\ra$.


$ $\\
\noindent\textbf{Claim 1.} {\em Let $x\in X$. Then $F_x\circ F_{i(x)}$ agrees with the identity map on a large subset of $X$.}
\begin{proof}[Proof of Claim 1] By Axiom (3)(b).
\end{proof}


 We may thus denote  $\bbm{1}_K=[x, i(x)]\in K$, for any $x\in X$.


$ $\\
\noindent\textbf{Claim 2.}
{\em For every $a, b\in X$, if
$$T=[(a,b)]=\{(x,y)\in X^2:  (x,y)\sim (a,b)\},$$
 then $\pi_1(T)$ and $\pi_2(T)$ are large in $X$.
  }
\begin{proof}[Proof of Claim 2] Pick $z\in X$ such that $F(a, F(b,z))$ is defined. Let $x=F(a, F(b,z))$. Now the set
$$C:= F( Z_z, z)\cap \pi_2 F^{-1}(x) $$ is large, by Axioms (2) and (4). Pick any $e\in C$. Let $c=F(-, e)^{-1}(x)$. Then $F(c,e)=x$ and there is $d\in Z_z$ such that $e=F(d, z)$. That is,
\begin{equation}
  F(c, F(d, z))=F(a, F(b, z)).\label{Fcdz}
\end{equation}
Clearly, for $e'\ne e$ in $S$, the corresponding $c', d'$ satisfy $c'\ne c$ and $d'\ne d$, by injectivity of $F$ in the first coordinate. Hence, we obtain large many such $c$'s and large many such $d$'s. We will be done if we prove that for every two such $c, d$,
\begin{equation}
  F_c\circ F_d \text{ agrees with } F_a\circ F_b \text{ on a large subset of } X.\label{Fcd}
\end{equation}
Fix any two $c,d\in X$ such that (\ref{Fcdz}) holds. Consider the set $A$ of all those $t\in X$, such that all expressions below are defined and are equal:
$$c(d(zt))=c((dz)t)=(c(dz))t=(a(bz))t=a((bz)t)=a(b(zt)).$$
By  Axiom (3)(a), the set $A$ is large in $X$.
Hence, using again injectivity of $F$ in the first coordinate, equality (\ref{Fcd})  holds for large many $zt$'s, as required.
\end{proof}

$ $\\
\noindent\textbf{Claim 3.}
(1)  {\em For every $a,b,c,d\in X$, there are $e, x, y, f\in X$, such that $(a, b)\sim (e, x)$, $(c, d)\sim (y, f)$ and $y=i(x)$.}

 (2)  {\em For every $a,b, c,d,s,t\in X$, there are $e,x,y,z,w,f\in X$, such that $(a, b)\sim (e, x)$, $(c, d) \sim (y,z)$, $(s, t) \sim (w,f)$, $y=i(x)$ and $w=i(z)$.}
\begin{proof}[Proof of Claim 3] We only prove (1), as the proof of (2) is similar.
Consider the sets
$$S=[(a,b)]=\{(e,x)\in X^2 : (a, b) \sim (e,x)\}$$
and
$$T=[(c,d)]=\{(y,f)\in X^2: (c, d) \sim (y, f)\}.$$
By Claim (2), the projections $\pi_1(T)$ and $\pi_2(S)$
are large in $X$. By Axiom (1), $\dim i(\pi_2(S))=k$, so the set
$$i(\pi_2(S))\cap \pi_1(T)$$
is non-empty. Take any $y$ in this set and  let $x\in \pi_2(S)$ with $y=i(x)$.  Then $y\in \pi_1(T)$ and  $x\in \pi_2(S)$, so there are $e,f\in X$ with $(a,b)\sim (e,x)$ and $(c,d)\sim (y,f)$, as needed.
\end{proof}

Now, for every $a,b,c,d, e,f\in X$, define the relation
\begin{align}
  R(a,b,c,d,e,f)\,\,\Lrarr \, & \text{ there are $x,y\in X$ such that $(a, b)\sim (e, x)$, $(c, d)\sim (y, f)$}\notag\\ &\text{ and $y=i(x)$.}\notag
\end{align}
Clearly, $R$ is a definable relation. By Claim 3(1), for every $a,b,c,d\in X$, there are $e,f\in X$ such that $R(a,b,c,d,e,f)$. Moreover, if $R(a,b,c,d,e,f)$, \linebreak$R(a',b',c',d',e',f')$, $(a, b)\sim (a', b')$ and $(c, d)\sim (c', d')$, then $(e, f) \sim (e', f')$. Indeed, let $x,y, x',y'$ witnessing the first two relations.
Then,  by Axiom (3), the set of $z\in X$ for which the following hold
$$e(fz)=e(x(y(fz)))=a(b(c(dz)))=a'(b'(c'(d'z)))=e'(x'(y'(f'z)))=e'(f'z).$$
is large in $X$. 
Hence,  $(e,f)\sim (e', f')$.
We can thus define the following definable operation on $K$:
$$[a,b]*[c,d]=[e,f]\,\,\Lrarr\,\, R(a,b,c,d,e,f).$$

$ $\\
\noindent\textbf{Claim 4.} {\em $K=\la K, *, \bbm{1}_K\ra$ is a definable group.}
\begin{proof}[Proof of Claim 4]
We already saw that the set $K$ and  map $*$ are definable. We prove associativity of $*$. Let $a,b,c,d,s,t\in X$. Take $e,x,y,z,w,f$ as in Claim 3(2). Then, by tracing the definitions, \begin{align}
  ([a,b]*[c,d]) *[s,t]=[e,z]*[w,f]=[e,f] &=[e,x]*[y,f]   \notag\\
  &= [a,b]*([c,d] *[s,t]).\notag
\end{align}
It is also easy to check that $\bbm{1}_K$ is the identity element, using Claim 3(1), and that $[i(b), i(a)]$ is the inverse of $[a,b]$.
 \end{proof}

Consider the map $h:X\to K$ defined as follows. Let $x\in X$. By Axiom (4), there is in particular $(a,b)\in Z$ such that $F(a,b)=x$. By definition of $\sim$, any two such $(a,b)$ are $\sim$-equivalent. We let
$$h(x)=[a,b].$$
Again by definition of $\sim$, it is clear that $h$ is injective.

$ $\\
\noindent\textbf{Claim 5.} {\em The set $h(X)$ is large in $K$.}
\begin{proof}[Proof of Claim 5]
If not, there must be a set $S\sub K\sm h(X)$ of dimension $k$ ($=\dim X$). By  Claim 2, each $\sim$-class has dimension $k$. Therefore, the union
$$T=\bigcup_{[a,b]\in S}[(a,b)]$$
has dimension $2k$. Since $Z$ is large in $X^2$, it must intersect $T$. But for $(c,d)\in Z\cap T$, we have $F(c,d)\in X$, and hence $[c,d]=h(F(c,d))\in h(X)$, contradicting the fact that $[c,d]\in S$.
\end{proof}

Finally, we check that for every $(s,t)\in Z$,
\begin{equation}
  h(F(s,t))=h(s) * h(t).\label{h}
\end{equation}
Let $a,b,c,d\in X$ such that $s=F(a,b)$ and $t=F(c,d)$. As in the proof of Claim 3 (1), we can find $e,x,y,f\in X$ with $F(e,x)=s$ and $F(y,f)=t$ and $y=i(x)$. It follows in particular that
$$h(s)*h(t)=[a,b]*[c,d]=[e,f].$$
Moreover, by Axiom (3), we obtain that for large many $z\in X$,
$$(st)z=((ex)(yf))z=(ex)((yf)z)=(ex)(y(fz))=e(x(y(fz))) =e(fz) =(ef)z.$$
Therefore, using Axiom (2), we obtain $st=ef$.
Hence
$$h(F(s,t))=h(F(e,f))=[e,f],$$
as required.\\


This ends the proof of (i).
For (ii) and (iii), in order to lighten the notation and without loss of generality, we may assume that $\sigma$ is the identity map.\smallskip

\noindent (ii)  We have, for every $(x,y)\in Z$,
$$x*y=F(x,y)=x \oplus y.$$
It follows that for every $a,b, c,d\in X$,
$$a*b=c*d \,\,\Lrarr\,\, a\oplus b=c\oplus d.$$
 Indeed, take $z\in X$, such that all pairs $(b,z), (a, b*z), (a, b\oplus z)$ are in $Z$, and hence
$$a*(b*z)=a\oplus (b\oplus z).$$
Therefore
$$a*b*c=z*d*z\,\,\Lrarr\,\,\ a\oplus b\oplus z=c\oplus d\oplus z,$$
 as needed.

Denote now by $^{-1}$ the inverse map of $K$.
We observe that for every $x\in K$, there are $a,b\in X$, such that $a*b=x$ (indeed, since $X$ is large in $K$, we can choose $a\in X\cap x*X^{-1}$ and $b=a^{-1} * x$). We can thus define the injective map
$$\text{ $\tau: K\to H$ given by $a*b\mapsto a\oplus b$, where $a,b\in X$}.$$
It remains to see that $\tau$ is a group homomorphism. The proof is similar to the one of property (\ref{h}) of $h$ above and is left to the reader.\\

\noindent (iii)  Since $X$ is large in $H$, for every $a\in H$, there is $z\in X$, such that both $a\ominus z$ and $z$ are in $X$, implying that  $H=X\oplus X$. Hence, in this case, the homomorphism $\tau$ constructed above is also onto.\smallskip
\end{proof}


\section{Weakly o-minimal structures}\label{sec-wom}




Here we return to the setting of the introduction, where $\cal M=\la M, <, +, \dots\ra$ is a weakly o-minimal non-valuational expansion of an ordered group, and $\cal N=\la N, <, +, \dots\ra$ its canonical o-minimal extension, by Wencel \cite{wen-nonv}. We recall that $\CN$ is an o-minimal structure whose domain $N$ is the Dedekind completion of $M$, and whose induced structure on $M$ is precisely $\CM$. The precise construction of $\cal N$ from \cite{wen-nonv} will not play any particular role here. We will only need the following facts, which follow easily from the strong cell decomposition theorem proved in \cite{wen-nonv}.

\begin{fact}[Wencel \cite{wen-nonv}]\label{fact-wen} We have:
\begin{enumerate}
\item Let $X$ be a definable set. Then $\cl(X)$ is $\cal N$-definable.
   \item  Suppose $f:X\sub M^n\to M$ is an definable map. Then there is an $\CN$-definable  map $F:N^n\to N$ that extends $f$ (namely, $F_{\res X}=f$).


 \item $\dim X=\dim \cl(X)$.
 \end{enumerate}
\end{fact}



We  need some additional terminology:  if $X, Z\sub N^n$, we call $X$ dense in $Z$ if $Z\sub \cl(X\cap Z)$. Equivalently, $X$ intersects every relatively open subset of $Z$. We write $X\triangle Z=(X\sm Z)\cup (Z\sm X)$. By a $k$-cell we mean a cell in $\cal N$ of dimension $k$.


Our goal is to establish Proposition \ref{findchunk} below, from which Theorem \ref{main} will follow. We first need to ensure that injective definable maps always extend to injective $\cal N$-definable maps. We prove this uniformly in parameters. In the next lemma, $\pi_1:M^{k+n}\to M^m$ denotes the projection onto the first $k$ coordinates.

\begin{lemma}\label{lem-inj}
Let $f:A\sub M^{k+n}\to M^m$ be a definable map, and $F: N^{k+n}\to N^m$ an $\CN$-definable extension. Assume that for every $t\in \pi_1(A)$, $f_t: A_t\to M^m$ is injective. Then there is an $\cal N$-definable set $A\sub Y\sub N^{k+n}$, such that for every $t\in \pi_1(Y)$, $F_t: Y_t\to N^m$ is injective.
\end{lemma}
\begin{proof}
By Fact \ref{fact-wen}(2), there is an $\CN$-definable family of functions $G_t: N^n\to N$, $t\in N^k$, such that for every $t\in \pi_1(A)$, $G_t$  extends $f^{-1}_t$. Let
$$Y_t=\{x\in N^n: G_t\circ F_t(x)=x\}.$$
Clearly, $A_t\sub Y_t$. Moreover,  for every $x,y\in Y_t$, if $F_t(x)=F_t(y)$, then
$$x=G_t\circ F_t(x)= G_t\circ F_t(y)=y,$$
and hence ${F_t}_{\res Y_t}$ is injective.
\end{proof}

For the rest of this section we fix a definable group $G=\la G, \cdot, \bbm 1_G\ra$ with $G\sub M^n$. Denote by $^{-1}$ its inverse.

 \begin{cor}\label{cor-inj} There are
\begin{enumerate}
   \item an $\cal N$-definable $i:N^n\to N^n$ that extends $^{-1}$,
   \item an \cal N-definable $F:N^{2n}\to N^n$ that extends $\cdot$,

\item  an $\cal N$-definable set $X\sub N^n$ containing $G$, such that $i_{\res X}$ is injective, and
\item an \cal N-definable set $Z\sub N^{2n}$ containing $G^2$, such that
\begin{itemize}
  \item for every $x\in \pi_1(Z)$, the map $F(x, -): Z_x\to F(x, Z_x)$ is injective,
  \item for every $x\in \pi_2(Z)$, the map $F(-, x):Z^x\to F(Z^x, x)$ is injective.
\end{itemize}
\end{enumerate}
 \end{cor}
 \begin{proof} By Fact \ref{fact-wen}(2), there are $i$, $F$ as in (1)-(2). By Lemma \ref{lem-inj}, (3)-(4) follow.
 \end{proof}

We can now achieve the connection to Section \ref{sec-gpchunk}.

\begin{prop}\label{findchunk} There is an $\cal N$-definable group chunk $\la X, i_{\res X}, F_{\res Z}\ra$, such that $G^2\sub Z$, $G\sub X\sub \cl(G)$, 
  $F_{\res G^2}=\cdot$, and $i_{\res G}= ^{-1}$.
\end{prop}
\begin{proof} Let $i, F, X, Z$ be as in Corollary \ref{cor-inj}. Let $X_1=\cl(G)\cap X$, and denote $Z_1=Z$.

For every $x\in G$, the set of $z\in X_1$ for which Axiom (3)(b) holds contains $G$, and hence, since $X_1\sub \cl(G)$, it is large in $X_1$. Let  $X_2$ be the set of all $x\in X_1$, for which the set of $z\in X_1$ such that Axiom 3(b) holds is large in $X_1$. Hence $G\sub X_2$.

Similarly, for every $x\in G$, $\pi_1 F^{-1}(x)$ and $\pi_2 F^{-1}(x)$ contain $G$, and hence are large in $X_1$. Let
$$X_3=\{x\in X_1: \text{$\pi_1 F^{-1}(x)$ and $\pi_2 F^{-1}(x)$ are large in $X_1$}\}.$$
So $G\sub X_3$. Therefore, for $X=X_2\cap X_3\cap \cl(G)$, we have  $G\sub X\sub \cl(G)$.

Finally, for every $(x,y)\in G^2$, the set of $z\in X$ for which Axiom (3)(a) holds contains $G$ and hence is large in $X$. Let $Z$ be the set of tuples $(x,y)\in Z_1\cap X^2$ for which  the set of $z\in X$ such that Axiom (3)(a) holds is large in $X$. Hence $G^2\sub Z$.

It is then straightforward to check that $X, i_{\res X}, F_{\res Z}$ are as required.
\end{proof}

We are now ready to prove the main result of this paper.


\begin{proof}[Proof of Theorem \ref{main}]
Let $X, Z, F, i$ be as in Proposition \ref{findchunk}, and $\la K, *\ra$ the $\cal N$-definable group from Theorem \ref{thm-gpchunk}. Then $G\leqslant K$. Indeed, for every $(x,y)\in G^2\sub Z$, $x\cdot y=F(x,y)=x*y$. We show that  for any $\cal N$-definable group $H$ that contains $G$ as a subgroup, $K$ $\cal N$-definably embeds in $H$. Consider $Z'=H^2 \cap Z$ and $X'=\pi_1(Z')\cap \pi_2(Z')$. Then Proposition \ref{findchunk} also holds with $X', Z', F, i$. Apply Theorem \ref{thm-gpchunk}(i) to get a group $K'$. Observe moreover that $X'$ is large in $K$.
Now, on the one hand, by Theorem \ref{thm-gpchunk}(iii) applied to $K'$ and $K$, we obtain that the two are $\cal N$-definably isomorphic. On the other hand, applying Theorem \ref{thm-gpchunk}(ii) to $K'$ and $H$, we obtain that $K'$ $\cal N$-definably embeds in $H$. Therefore $K$ $\cal N$-definably embeds in $H$.


For the ``moreover" clause, take a $K$-open subset $U$ of $K$. Then $\dim U =\dim K$.  By Theorem \ref{thm-gpchunk}(i), $X$ is large in $K$, and hence $\dim (U\cap X)=\dim X$. By Proposition \ref{findchunk}, $G$ is dense in $X$, and hence $U\cap X\cap G\ne \es$, showing $\cl_K(G)=K$. We also have $\dim G= \dim \cl(G)=\dim X=\dim K$.
\end{proof}

 \begin{remark}\label{rmk-gtop}
It follows from Theorem \ref{main} that $G$ admits a group topology and contains a large subset on which the group topology coincides with the subspace one. Indeed, the restriction of the $K$-topology to $G$ is a group topology, and if $V$ is a large subset of $K$ on which the $K$-topology coincides with the subspace one, then  $V\cap G$ also has the same properties with regard to $G$, as can easily be shown. As mentioned in the introduction, these results  already follow from \cite[Theorem 4.10]{wen-gps}.
\end{remark}

\begin{remark}\label{rem-sym} Since the subspace and the group topologies coincide on a large subset $V$ of $K$, it is easy to see that, for every $X\sub K$,
$$\dim (\cl_K(X)\triangle \cl(X))<\dim K.$$
Indeed, since $V$ is large,  we may assume $X\sub V$. But then the relative closure of $X$ in $V$ in either topology is the same, whereas the remaining parts of the closures have dimension $<\dim K$.
\end{remark}

\begin{remark}\label{rem-can1} It is also worth noting that   $K$ is the smallest $\cal N$-definable group containing $G$, even up to definable isomorphism. Namely, for every  definable group isomorphism $\rho:G\to \rho(G)\sub H$ and $\cal N$-definable group $H$, there is an $\cal N$-definable embedding $f:K\to H$. Indeed, by Lemma \ref{lem-inj}, $\rho$ extends to an injective $\cal L$-definable map $\sigma:X\to \sigma (X)$, with $G\sub X\sub K$.  We can pullback the multiplication and inverse of $K$ (restricted to $\sigma(X)^2$ and $\sigma(X)$) to $X^2$ and $X$, respectively. Call those pullbacks $F$ and $i$. It is then not hard to see that $X, Z=X^2, F, i$ satisfy the assumptions of Theorem \ref{thm-gpchunk}, and, moreover, those of (ii) therein. Hence $K$ $\cal L$-definably embeds in $H$, as needed.
\end{remark}




We finish this section with some open questions. By \cite{wen-nonv}, every definable set is the trace of an $\CN$-definable set; that is, of the form $Y\cap M^n$, for some $\cal N$-definable set $Y\sub N^n$. It is therefore natural to ask the following question.

\begin{question} In Theorem \ref{main}, can $K$ be chosen so that, moreover,  $K\cap M^n=G$?
\end{question}

In general, one can ask for what structures $\cal R_1\preccurlyeq \cal R_2$ the following  statement is true: given a group $G=R_2^n \cap X$, for some $\cal R_2$-definable set $X$, is $G$ a subgroup of an $\cal R_2$-definable group $K$, such that: (a) $K$ is the smallest such group,  (b) $G=R_2^n \cap K$?\vskip.1cm


Finally, the following question was asked by Elias Baro in private communication.

\begin{question} Assume $\cal M$ is the trace of a real closed field \cal N on a dense  $\cal R\preccurlyeq \cal N$. Is every definable group $\la \cal N, \cal R\ra$-definably isomorphic to a group definable in $\cal R$?
\end{question}

\section{Pillay's Conjecture}\label{sec-app}

In Section \ref{sec-pi} below, we will assume further that $\cal M$ is sufficiently saturated, and prove Theorem \ref{app} by means of Claims \ref{G00}, \ref{G/G00} and \ref{cdom} below. It is important for us that the canonical extension \cal N is presented in some signature so that it also becomes sufficiently saturated (and hence Pillay's Conjecture for $\cal N$-definable groups holds). Towards that end, we adopt the following definition from \cite{bhp}. Denote by \cal L the language of \cal M.

\begin{defn}
$\cal M^*$ is the \cal L-structure whose domain is the Dedekind completion $N$ of $M$, and for every $n$-ary $R\in \cal L$, we interpret $R^{\cal M^*}=\cl(R^{\cal M})$.
\end{defn}

\begin{fact}\label{fact-M*} We have:
\begin{enumerate}
\item $\cal M^*$ and \cal N have the same definable sets.
\item \cal M and $\cal M^*$ have the same degree of saturation.
\end{enumerate}
\end{fact}
\begin{proof}
(1) Clearly, $\cal M^*$ is a reduct of \cal N. For the opposite direction, note that in \cite{bhp}, $\cal M^*$ is denoted by $\cal M^*_\es$, and $\cal N$ by $\overline{\mathcal M}_M$. The conclusion then follows from \cite[Theorem 2.10, Proposition 2.7 and Proposition 3.4(2)]{bhp}.

(2) By \cite[Proposition 3.4(2)]{bhp}.
\end{proof}

In view of Fact \ref{fact-M*}(1), everything proven for \cal N in the previous sections still holds for $\cal M^*$. We may thus, from now on, let \cal N denote $\cal M^*$.


\subsection{Preliminaries} In this subsection, we fix a sufficiently saturated structure \cal R, and definability is taken with respect to \cal R. For a definable group $G$, the notion of having \emph{finitely satisfiable generics} ($fsg$) was introduced in \cite{hpp}, generalizing the notion of definable compactness from o-minimal structures. By now, several other statements involving Keisler measures have been shown to be equivalent to $fsg$. Below, we adopt as a definition of $fsg$ a statement that involves \emph{frequency interpretation measures}, first introduced in \cite{hps2}. Our account uses \cite{simon-book} and  \cite{star}.

For the next three definitions, $G$ denotes a definable group with $G\sub R^n$.

\begin{defn}
A \emph{Keisler measure} $\mu$ on $G$ is a finitely additive probability measure on the class $Def(G)$ of all definable subsets of $G$; that is, a map $\mu: Def(G)\rarr [0,1]$ such that
$\mu(\emptyset)=0$, $\mu(G)=1$, and for $Y,Z\in Def(G)$, $$\mu(Y\cup Z)=\mu(Y)+\mu(Z)-\mu(Y\cap Z).$$
\end{defn}

Given a definable set $X\sub R^n$ and elements $a_1, \dots, a_k\in R^n$, we denote
$$\mathrm{Av}( a_1, \dots,  a_k;X)|= \frac{1}{k}|\{i:  a_i\in X\}|.$$

 \begin{defn}
A Keisler measure $\mu$ on $G$ is called  a \emph{frequency interpretation measure ($fim$)}  if  for  every  formula $\varphi(x;y)$ and  $\ve >0$, there are $a_1,\dots, a_k\in R^n$, such  that  for  every  $c\in R^m$, and for  $X=\varphi(R;  c)$, we have $$|\mu(X)-\mathrm{Av}( a_1, \dots, a_k;X)|< \ve.$$
 \end{defn}

\begin{defn}\label{def-fsg}
We say that $G$ has \emph{finitely satisfiable generics} ($fsg$) if it admits a left-invariant $fim$ Keisler measure.
\end{defn}

\begin{remark}
Our definition of $fsg$ is equivalent to the original definition given in \cite{hpp}. Indeed, by \cite[Theorem 3.32]{star}, $\mu$ has $fim$ if and only if it is `generically stable', and by \cite[Proposition 8.33]{simon-book}, $G$ has $fsg$ in the sense of \cite{hpp} if and only if $G$ admits a generically stable Keisler measure.
\end{remark}


Recall from the introduction that by a \emph{small} set or a set of \emph{bounded cardinality}, we mean a set of cardinality smaller than $|R|$. By a type-definable set $X$ we mean an intersection of a bounded collection of definable sets $X_i$. We write 
$X=\bigcap_i X_i$ without specifying the index set.

For the next two statements, we assume that $\cal R$ is our fixed \cal M or \cal N. For any $X\sub R^n$, $\dim X$ is the maximum $k$ such that some projection onto $k$ coordinates contains an open set. This notion is the same with the one defined in the introduction.


\begin{fact}\label{bddindex}
Let $H$ be a type-definable subgroup of a definable group $G$. If $H$ has bounded index, then $\dim H=\dim G$.
\end{fact}
\begin{proof}
Easy, by compactness. 
\end{proof}

\subsection{The proof of Pillay's Conjecture}\label{sec-pi} From now on, we assume that \cal M is sufficiently saturated. By Fact \ref{fact-M*}, $\cal N$ is also sufficiently saturated. We also fix a definable group $G$ and the $\cal N$-definable group $K$ from Theorem \ref{main}. We write $ab$ for the multiplication in $K$ and $G$. We let $k=\dim G=\dim K$. 
By `type-definable' we mean `type-definable in $\CM$', and by `$\CN$-type-definable' we mean `type-definable in $\CN$'. 
As mentioned in the introduction, since $\cal M$ has NIP,  $G^{00}$ exists. Here, however, we provide a more precise description of $G^{00}$ which further enables us to prove Claim \ref{G/G00} below.

\begin{claim}\label{G00}
$G$ has a smallest type-definable subgroup $G^{00}$ of bounded index. More precisely,  $G^{00}=G\cap K^{00}$.
\end{claim}
\begin{proof} We show that $G\cap K^{00}$ is the smallest type-definable subgroup of $G$ of bounded index. It is certainly a type-definable subgroup of bounded index, since $[G:G\cap K^{00}]\le [K:K^{00}]$. Let $H$ be another such subgroup. We prove that $H$ must contain $G\cap K^{00}$. Let $L=\cl_K(H)$.
We claim that $L$ is a type-definable subgroup of $K$ of bounded index, and hence it must contain $K^{00}$. It is certainly a subgroup, since it is the closure of a subgroup of $K$.

To see that $L$ is type-definable, let $H=\bigcap_i H_i$, with $H_i$ definable. We may assume that the family $\{H_i\}$ is closed under finite intersection. Then $\cl_K(H)=\bigcap_i\cl_K(X_i)$. Indeed, for $\sub$, the right-hand side is a closed set containing $H$ and hence also $\cl_K(H)$. For $\supseteq$, let $a\in\bigcap_i\cl_K(X_i)$. Then for every $i$, there is an open box $B$ containing $a$, with $B\cap X_i=\es$. By compactness, there is an open box $B$ containing $a$, with $B\cap \bigcap_i X_i=\es$. That is, $a\not\in \cl_K(H)$.

To see that $L$ has bounded index in $K$, note that since $H$ has bounded index in $G$, by Fact \ref{bddindex} we have $\dim H= k$. Hence $\dim \cl(H)=k$. By Remark \ref{rem-sym},   $\dim \cl_K(H)=k= \dim K$. 
By Fact \ref{bddindex} again, $L$ has bounded index in $K$.


We have shown that $K^{00}\sub L$. Hence, to prove that $G\cap K^{00}\sub H$, it suffices to show $H=G\cap L$. Clearly, $H$ is a subgroup of $G\cap L$. Assume, towards a contradiction, that $H$ is properly contained in $G\cap L$. Then there must be a coset $aH$  of $H$ in $G\cap L$, such that $aH\sub (G\cap L)\sm H$. Since $\dim aH=\dim H=\dim L$, and $\cl_K(H)=L$, it follows 
 that $aH\cap H\ne \es$, a contradiction.
\end{proof}



By Fact \ref{bddindex},  $G^{00}$ has non-empty interior in the $G$-topology, and hence it follows that it is open in the $G$-topology. 

\begin{claim}\label{G/G00} The function $f:G/G^{00}\to K/K^{00}$ given by $$x G^{00}\mapsto x K^{00}.$$ is an isomorphism of topological groups with respect to the logic topology on both groups.
In particular, $G/G^{00}$ is a compact Lie group. 
\end{claim}
\begin{proof}
By Claim \ref{G00}, $G^{00}=G\cap K^{00}$. We first prove that $G K^{00}= K$, which implies that $f$ is a group isomorphism, by the second isomorphism theorem. Let $x\in K$.
Since $\cl_K(G)=K$ and $\dim x K^{00}=\dim K^{00}=k$ (by Fact \ref{bddindex}), there must be $y\in x K^{00}\cap G$. Thus $x\in y K^{00}\sub GK^{00}$, as needed.

It remains to show that $f$ is a homeomorphism, with respect to the logic topology on both $G/G^{00}$ and $K/K^{00}$. By \cite[Lemma 2.5]{pi-conj}, both quotients are compact Hausdorff spaces, and hence, since $f$ is a bijection, it suffices to show that it is continuous.
Write $\pi: G\to G/G^{00}$ and $\sigma:K\to K/K^{00}$ for the canonical group homomorphisms. We need to show that for every $X\sub G/G^{00}$, 
$$\text{$\sigma^{-1}(f(X))$ is type-definable \,\,$\Rarr$\,\, $\pi^{-1}(X)$ is type-definable}.$$
Let
$$X=\{x G^{00} : x\in S\},$$
for some $S\sub G$. Then $\pi^{-1}(X)=S G^{00}$, $f(X)=\{x K^{00} :x\in S\}$, and
$\sigma^{-1}(f(X))= S K^{00}$. Hence we need to prove that
\begin{equation}\text{$S K^{00}$ is type-definable \,\,$\Rarr$\,\, $S G^{00}$ is type-definable.}\label{SKG}
\end{equation}
We have
$$S G^{00}=S (K^{00}\cap G)=(S K^{00}) \cap G,$$
since $S\sub G$. This implies (\ref{SKG}).
\end{proof}


We now turn to the property $fsg$.

\begin{lemma}\label{fsg}
Assume $G$ has $fsg$. Then so does $K$.
\end{lemma}
\begin{proof}
Since $G$ has $fsg$, there is a left-invariant $fim$ Keisler measure $\mu$ on $G$.  We define $\nu : Def(K)\to [0,1]$ by setting $\nu(X)=\mu(X\cap G)$, and show that $\nu$ is a left-invariant $fim$ Keisler measure. It is straightforward to see that it is a Keisler measure, since $\mu$ is.	
In what follows, all topological notions for subsets of $K$ 
are taken with respect to its group topology.\\

	\noindent\textbf{Claim 1.} {\em If $X\subseteq K$ is definable then  $\nu(X)=\nu(\intr(X))$, where $\intr(X)$ is the interior of $X$.}
	
	\begin{proof}[Proof of Claim 1]
	  It suffices to show that if $X$ has empty interior, then $\nu(X)=0$. For that we will show that if $S\subseteq G$ is definable with empty interior, then $\mu(S)=0$. The assumption that $S$ has empty interior implies that  $\dim(S)<\dim(G)$, so $S$ is non-generic. Since $G$ has $fsg$, $\mu(S)=0$ by \cite[Fact 3.1(iv)]{hps2}.
	\end{proof}
	
	\noindent\textbf{Claim 2.} {\em $\nu$ is left-invariant.}
\begin{proof}[Proof of Claim 2]
Let $X\subseteq K$ be a definable set and $k\in K$. We need to show that $\nu(kX)=\nu(X)$. By  Claim 1, we may assume that $X$ is open. Let $S=X\cap G$ and $S_k=kX\cap G$. 
Assume towards a contradiction that $\mu(S)>\mu(S_k)$ (the other inequality is treated similarly). 	Fix $\ve << \mu(S)-\mu(S_k)$.
By applying $fim$ of $\mu$ to the union of the families $\{gS : g\in G\}$ and $\{g S_k: g\in G\}$, we obtain $a_1,\dots, a_r\in M^n$ such that
$$|\mu(gS)-\mathrm{Av}( a_1, \dots, a_k; gS)|< \ve$$
and
$$|\mu(gS_k)-\mathrm{Av}( a_1, \dots, a_k;gS_k)|< \ve.$$
Assume, for simplicity of notation that $a_i\in S$ if and only if $i<j$ (some $j$). Note that $gS_k=gkX\cap G$ for $g\in G$. Continuity of the group operation, and the assumption that $X$ is open assure that $a_1,\dots, a_j\in hX$ for all $h\in K$ close enough to $e$. So for $g\in G$ close enough to $k^{-1}$, we get that $a_1,\dots a_j\in gkX\cap G=gS_k$. The choice of $a_1,\dots, a_r$ implies that
$\mu(gS)<\frac{j}{r}+\ve$ and $\mu(gS_k)>\frac{j}{r}-\ve$. Hence
$$\mu(gS)-\mu(gS_k)<2\ve.$$
Since $\mu$ is left-invariant, we obtain $\mu(S)-\mu(S_k)<2\ve$, contradicting the choice of $\ve$.
\end{proof}

\noindent\textbf{Claim 3.} {\em $\nu$ is $fim$.}

\begin{proof}[Proof of Claim 3]
Let $\varphi(x;y)$ be a formula. Then for every $b\in N^m$ the set $\varphi(M,b)$ is definable, say by $\psi(x,c)$ for some $c\in M^m$. Of course, if $b$ and $b'$ have the same type in the pair $\la \CN,\CM\ra$, then there exists $c'\in M^m$ such that $\varphi(M,b')=\psi(M,c')$. So by compactness there is a formula $\theta(x,z)$, such that for all $b\in N^m$, there is $c\in M^l$ with $\varphi(M,b)=\theta(M,c)$.  Now, since $\mu$ is $fim$, for every $\ve>0$, there are $a_1,\dots, a_k\in M^m$, such that for every $c\in M^l$, and for $X=\varphi(M^n, c)$, we have
$$|\mu(X)-\mathrm{Av}( a_1, \dots, a_k;X)|< \ve.$$
But since for every $b\in N^m$, there is $c\in M^l$ with $$\nu(\varphi(N,b))=\mu(\varphi(M,b)=\mu(\theta(M,c)),$$ the result follows.
\end{proof}
This ends the proof of the lemma.
\end{proof}

We can now conclude the proof of Theorem \ref{app}.

\begin{claim}\label{cdom} Suppose $G$ has $fsg$. Then $\dim G=\dim_{Lie} G/G^{00}$.
   \end{claim}
\begin{proof} 
By Lemma \ref{fsg}, $K$ has $fsg$.   By Pillay's Conjecture for o-minimal structures, $\dim_{Lie} K/K^{00}=\dim K=k$. By Claim \ref{G/G00}, we are done.
\end{proof}

We conclude with some  remarks.

\begin{remark}\label{rmk-cdom}
It is worth mentioning that in \cite{hpp},  following the proof of Pillay's Conjecture, the Compact Domination Conjecture was introduced, and  was proved by splitting different cases in \cite{el-cdom, elpet2, hpp2, hp}. 
Compact Domination for a NIP group $G$ turned out to be a crucial property, as it corresponds to the existence of an invariant smooth Keisler measure on $G$ (\cite[Theorem 8.37]{simon-book}). By now it is known that every $fsg$ group definable in a `distal' NIP structure (which includes weakly o-minimal structures) is compactly dominated (\cite[Chapters 8 \& 9]{simon-book}). The present account can actually yield a short proof of compact domination for fsg groups definable in \cal M, which we however omit.
\end{remark}


\begin{remark}\label{rmk-small}
The results of this paper also apply  to another category of definable groups, namely \emph{small groups} in certain dense pairs $\la \cal N, P\ra$. Those pairs include expansions of a real closed field by a dense elementary substructure $P$ or a dense multiplicative divisible subgroup with the Mann property. Following \cite{vdd-dense}, a definable set $X\sub N^n$ is called \emph{small} if there is an $\cal N$-definable map $f:N^{mk}\to N^n$ such that $X\sub f(P^k)$. By \cite{el-Pind} every small set is in definable bijection with a set definable in the induced structure on $P$, which is known to be weakly o-minimal and non-valuational (\cite{vdd-dense, dg}).
Therefore, we have established Pillay's Conjecture for small groups in dense pairs.
\end{remark}

\end{document}